\documentclass[a4paper,draft]{article}

\usepackage[latin1]{inputenc} 
\usepackage[T1]{fontenc}
\usepackage{amsmath}
\usepackage{amsthm}
\usepackage{amssymb}     
\usepackage{mathrsfs}    
\usepackage{enumerate}

\theoremstyle{plain}
\newtheorem{thm}{Theorem}[section]
\newtheorem{lem}[thm]{Lemma}
\newtheorem{prop}[thm]{Proposition}
\newtheorem{cor}[thm]{Corollary}

\theoremstyle{definition}

\newtheorem{rem}[thm]{Remark}

\newcommand{\coloneq}{\mathrel{\mathop:}=}

\newcommand{\st}{\,;\,}

\newcommand{\integers}{\ensuremath{\mathbb{Z}}}
\newcommand{\reals}{\ensuremath{\mathbb{R}}}
\newcommand{\complex}{\ensuremath{\mathbb{C}}}
\newcommand{\rd}[1]{\ensuremath{\reals^{#1}}}
\newcommand{\sphere}{\ensuremath{S^2}}

\newcommand{\lspace}[2]{\ensuremath{L^{#1}\left(#2\right)}}
\newcommand{\ltwob}{\ensuremath{\lspace 2 {b, b+1}}}
\newcommand{\ltwo}{\ensuremath{\lspace 2 {0, \infty}}}

\newcommand{\sobspace}[2]{\ensuremath{H^{#1}\left(#2\right)}}

\newcommand{\honeb}{\ensuremath{\sobspace 1 {b,b+1}}}
\newcommand{\htwob}{\ensuremath{\sobspace 2 {b,b+1}}}

\newcommand{\czeroinf}[1]{\ensuremath{C_0^{\infty}\!\left(#1\right)}}
\newcommand{\cz}{\ensuremath{\czeroinf{0,\infty}}}
\newcommand{\czthree}{\ensuremath{\czeroinf{\rd 3 \! \setminus \!\! \{0\}}}}

\newcommand{\hilbert}{\ensuremath{\mathscr{H}}}

\newcommand{\spec}[1]{\ensuremath{\sigma\left(#1\right)}}

\newcommand{\dom}[1]{\ensuremath{D\left(#1\right)}}
\newcommand{\qdom}[1]{\ensuremath{D\left[#1\right]}}

\DeclareMathOperator{\realpart}{Re}
\DeclareMathOperator{\supp}{supp}
\DeclareMathOperator{\linspan}{span}
\DeclareMathOperator{\rank}{rank}
\DeclareMathOperator{\tr}{tr}

\newcommand{\utrans}{\ensuremath{\mathscr{U}}}

\newcommand{\xp}{\ensuremath{\nu}}
\newcommand{\gc}{\ensuremath{\gamma_{\mathrm{c}}}}
\newcommand{\hardyconst}[2]{\ensuremath{C_{#1,#2}^{\mathrm{HR}}}}

\numberwithin{equation}{section}

\title{Lieb-Thirring Inequalities for Fourth-Order Operators in Low Dimensions}
\author{
  Tomas Ekholm\\
  {\small Department of Mathematics}\\
  {\small Lund University}\\
  {\small S-221 00 Lund, Sweden}\\
  {\small \texttt{tomas.ekholm@math.lu.se}}\\
  \and
  Andreas Enblom\\
  {\small Department of Mathematics}\\
  {\small Royal Institute of Technology}\\
  {\small S-100 44 Stockholm, Sweden}\\
  {\small \texttt{enblom@math.kth.se}}
}

\begin{document}
\maketitle

\begin{abstract}
  This paper considers Lieb-Thirring inequalities for higher 
order differential operators.
A result for general fourth-order operators on the half-line is
developed, and the trace inequality
\begin{displaymath}
  \mathrm{tr}\left(
    (-\Delta)^2 - C^{\mathrm{HR}}_{d,2} \frac{1}{|x|^4} - V(x)
  \right)_-^{\gamma} 
  \leq C_\gamma \int_{\mathbb{R}^d} V(x)_+^{\gamma + \frac{d}{4}} \, dx,
  \quad \gamma \geq 1 - \frac d 4,
\end{displaymath}
where $C^{\mathrm{HR}}_{d,2}$ is the sharp constant in the
Hardy-Rellich inequality and where $C_\gamma > 0$ is independent of
$V$, is proved for dimensions $d = 1,3$. As a corollary of this
inequality a Sobolev-type inequality is obtained.

\end{abstract}

%
%
%

\section{Introduction}
This paper concerns Lieb-Thirring inequalities with critical
exponents. 
Well-known results in this area are the Lieb-Thirring inequalities 
\begin{equation*}
  \tr \left((-\Delta)^{l} - V\right)_-^{\gamma}
  \leq C \int_{\rd d} V(x)^{\gamma + \frac{d}{2l}} \, dx,
  \quad \gamma \geq 1 - \frac{d}{2l},
\end{equation*}
in the space $\lspace 2 {\rd d}$, where $l > d/2$, as discussed in
\cite{weidl} and \cite{netrusov-weidl}. 

Recent papers such as
\cite{efcmp} and \cite{frank-lieb-seiringer} combine Lieb-Thirring
inequalities with the sharp Hardy-Rellich inequalities of the type
\begin{equation} \label{eq:introhardy}
  \int_{\rd d} \left| \nabla^{l} u(x) \right|^2 \, dx \geq
  \hardyconst{d}{l} \int_{\rd d} \frac{|u(x)|^2}{|x|^{2l}} \, dx,
\end{equation}
where $l < d/2$, as discussed in \cite{yafaev}. This case, however,
 does not admit a critical exponent. The inequalities
thus obtained are of the type
\begin{equation*}
  \tr \left((-\Delta)^{l} - \hardyconst{d}{l} \frac{1}{|x|^{2l}}
    - V(x)\right)_-^{\gamma}
  \leq C \int_{\rd d} V(x)_+^{\gamma + \frac{d}{2l}} \, dx,
  \quad \gamma > 0.
\end{equation*}

For inequalities with critical exponent, we again turn to the sharp
Hardy-Rellich inequality \eqref{eq:introhardy}, but this time we
assume that $l > d/2$ and $l - d/2 \notin \integers$. In this case the
inequality is valid for $u \in \czeroinf{\rd d \setminus \{0\}}$.
In \cite{efjems}, the following inequality is
obtained for the case $l = d = 1$:
\begin{equation*}
  \tr \left(-\frac{d^2}{dx^2} - \hardyconst{1}{1} \frac{1}{x^2} -
    V(x)\right)_-^{\gamma}
  \leq C \int_0^\infty V(x)_+^{\gamma + \frac 1 2} \, dx,
  \quad \gamma \geq \frac 1 2,
\end{equation*}
where the operator on the left-hand-side is taken with Dirichlet
boundary conditions at $0$.
In the present paper, we develop these techniques further to prove the
critical exponent inequality
\begin{equation} \label{eq:introineq}
  \tr \left((-\Delta)^{2} - \hardyconst{d}{2} \frac{1}{|x|^4} 
    - V(x)\right)_-^{\gamma}
  \leq C_\gamma \int_{\rd d} V(x)_+^{\gamma + \frac d 4} \, dx,
  \quad \gamma \geq 1 - \frac d 4, 
\end{equation}
for the fourth-order cases $l = 2$ and $d = 1,3$, where the constant
$C_\gamma > 0$ is independent of $V$. Again, the operator
in question is considered with Dirichlet conditions at $0$. 

In fact, we prove such an inequality for a general fourth-order 
operator on the half-line, from which the results for the bi-laplacian 
with Hardy weight in dimensions $d=1,3$ follow. This way we actually get 
a more general result than \eqref{eq:introineq} in the case $d=1$, by
introducing a weight in the integral on the right-hand side. Such
weighted inequalities exist for any $\gamma > 0$, and the weight can
be chosen such that $\gamma$ is still the critical exponent.

The methods for proving this general result origin in \cite{weidl},
\cite{netrusov-weidl} and \cite{efjems}. 
In this paper no Sturm-Liouville or Green's function theory is needed.
One interesting technical result is the Sobolev-type inequality of Lemma
\ref{lem:supu}.

It is worth noting that the proofs employed here 
can be extended to higher order $l \geq 3$ 
and dimensions $d$ such that $l > d/2$ and $l - d/2 \notin \integers$, 
even if this would be somewhat tedious.

Finally, an immediate consequence of inequality \eqref{eq:introineq}
is a Sobolev-type inequality that estimates the $L^p$-norm of a
function $u \in \czeroinf {\rd \setminus \{0\}}$, for $1 < p \leq \infty$.

\section{Main Results}

We prove trace inequalities in
dimensions $d=1,3$ for the fourth-order operator
\begin{equation*}
  H = H_0 - V,
  \quad \text{ where } \quad
  H_0 = (-\Delta)^2 - \hardyconst{d}{2} \frac{1}{|x|^4}.
\end{equation*}

\begin{thm} \label{thm:main1d}
  Let $0 \leq \xp < 3$ and $\gamma \geq (3-\xp)/4$. Then, for
  any non-negative $V$ such that $V(x)^{\gamma + (1+\xp)/4} x^{\xp}$ is
  integrable on $(0,\infty)$, the form
  \begin{equation*}
    u \mapsto
    \int_0^\infty \left(
      \left|u''(x)\right|^2
      - \hardyconst{1}{2} \frac{|u(x)|^2}{x^4}
      - V(x)|u(x)|^2
    \right)  \, dx,
  \end{equation*}
  is lower semi-bounded on $\cz$. 
  Let $H_0 - V$ be the self-adjoint operator corresponding to the
  closure of this form.
  Then the negative spectrum of $H_0-V$ is discrete, 
  and there is a constant
  $C = C(\xp, \gamma) > 0$, independent of $V$, such that
  \begin{equation*}
    \tr (H_0-V)_-^\gamma 
    \leq C \int_0^\infty V(x)^{\gamma + \frac{1+\xp}{4}} x^{\xp} \, dx.
  \end{equation*}
\end{thm}

The weight in the integral on the right-hand side is important for two
different reasons. First of all, it allows us to
consider arbitrarily small $\gamma > 0$. Second, it allows us to pass to
higher dimensions, as seen in Lemma \ref{lem:3d-firstpart}
which is an important part of the proof of the following theorem:

\begin{thm} \label{thm:main3d}
  Let $\gamma \geq 1/4$. 
  For any non-negative $V \in \lspace {\gamma + 3/4} {\rd 3}$, the form
  \begin{equation*}
    u \mapsto
    \int_0^\infty \left(
      \left|\Delta u(x)\right|^2
      - \hardyconst{3}{2} \frac{|u(x)|^2}{|x|^4}
      - V(x)|u(x)|^2
    \right)  \, dx,
  \end{equation*}
  is lower semi-bounded on $\czthree$. 
  Let $H_0 - V$ be the self-adjoint operator corresponding to the
  closure of this form.
  Then the negative spectrum of $H_0-V$ is discrete, 
  and there is a constant
  $C = C(\gamma) > 0$, independent of $V$, such that
  \begin{equation*}
    \tr (H_0-V)_-^{\gamma}
    \leq C \int_{\rd 3} V(x)^{\gamma + \frac{3}{4}} \, dx.
  \end{equation*}
\end{thm}

In fact, these results follow from a result for a class of general 
fourth-order Schr\"odinger operators on the half-line. The proof of this result 
captures all the essential ideas in this paper, and can be used to
prove similar results for other fourth-order operators than the
bi-laplacian with a Hardy term.
\begin{thm} \label{thm:maingeneral}
  Let $\alpha \geq 0$, $\beta \geq 0$, $0 \leq \xp < 3$, 
  $\xp \leq 2\beta$ and $\gamma
  \geq (3-\xp)/4$. Then, for
  any non-negative $V$ such that $V(x)^{\gamma + (1+\xp)/4} x^{\xp}$ is
  integrable on $(0,\infty)$, the form
  \begin{equation*}
    u \mapsto
    \int_0^\infty \left( \left|
        \frac{d}{dx} \left(
          \frac{1}{x^\alpha}
          \frac{d}{dx} \left(\frac{u(x)}{x^\beta}\right)
        \right)
      \right|^2 x^{2(\alpha+\beta)}
      - V(x)|u(x)|^2
    \right)  \, dx,
  \end{equation*}
  is lower semi-bounded on $\cz$. 
  Let $H_0 - V$ be the self-adjoint operator corresponding to the
  closure of this form.
  Then the negative spectrum of $H_0-V$ is discrete, 
  and there is a constant
  $C = C(\alpha, \beta, \xp, \gamma) > 0$, independent of $V$, such that
  \begin{equation*}
    \tr (H_0-V)_-^\gamma 
    \leq C \int_0^\infty V(x)^{\gamma + \frac{1+\xp}{4}} x^{\xp} \, dx.
  \end{equation*}
\end{thm}
The proofs of these three theorems are postponed 
until Section \ref{sec:mainproof}

An immediate consequence of these theorems is the following
Sobolev-type inequality:
\begin{cor}
  Let $d = 1,3$ and $1 < p \leq \infty$.
  Then there are constants $D_1, D_2 > 0$ such that
  \begin{equation*}
    \left( \int_0^\infty \!\!\!\! |u|^{2p} dx \right)^{\frac 1 p}
    \!\! \leq \! \int_0^\infty \!\!\! \left(
      \left|u''(x)\right|^2
      \!\! - \! D_1 \frac{|u(x)|^2}{x^4}
      \! + \! D_2 |u(x)|^2 \!
    \right)\! dx,
    \quad  u \in \czeroinf{\rd d \! \setminus \!\! \{0\}}\!.
  \end{equation*}
\end{cor}
\begin{proof}
  We prove only the case $d = 1$, as the other case is similar.
  Let $q \geq 1$ be such that $p^{-1} + q^{-1} = 1$. 
  Setting $\xp = 0$ and $\gamma = q - 1/4$ in Theorem \ref{thm:main1d}
  and letting
  \begin{equation*}
    E(V) = \inf \spec {H_0 - V}
  \end{equation*}
  we obtain that
  \begin{equation*}
    E(V)^{\gamma} \geq -C \|V\|_q^q \, dx,
  \end{equation*}
  for any non-negative $V \in \lspace q {0,\infty}$, where 
  $C > 0$ is independent of $V$. It follows that
  \begin{equation*}
    H_0 - V + C^{1/\gamma} \|V\|_q^{q/\gamma} \geq 0,
  \end{equation*}
  and hence
  \begin{equation} \label{eq:corpos}
    \int_0^\infty \left(
      \left|u''(x)\right|^2
      - \hardyconst{1}{2} \frac{|u(x)|^2}{x^4}
      - V(x)|u(x)|^2
      + C^{1/\gamma} \|V\|_q^{q/\gamma} |u(x)|^2
    \right)  \, dx \geq 0,
  \end{equation}
  for any $u \in \cz$ and any non-negative $V \in \lspace q
  {0,\infty}$. 
  Fix $u \in \cz$ and consider the linear
  functional $L_u : \lspace q {0,\infty} \to \complex$ given by
  \begin{equation*}
    L_u V = \int_0^\infty V|u|^2 \, dx.
  \end{equation*}
  Choose $V \in \lspace p {0,\infty}$ with $\|V\|_q = 1$
  and write $V = V_R +
  iV_I$, where $V_R$ and $V_I$ are
  real-valued. Furthermore, let $V^+_R$ and $V^-_R$ be the
  positive and negative parts of $V^R$, respectively. Define 
  $V^+_I$ and $V^-_I$ similarly.
  Note that $\|V^+_R\|_q \leq \|V\|_q = 1$ and hence by \eqref{eq:corpos},
  \begin{equation*}
    L_u V^+_R \leq 
    \int_0^\infty \left(
      \left|u''(x)\right|^2
      - \hardyconst{1}{2} \frac{|u(x)|^2}{x^4}
      + C^{1/\gamma} |u(x)|^2
    \right)  \, dx,
  \end{equation*}
  and similarly for $V_-^R$, $V_+^I$ and $V_-^I$. Hence $L_u$ is a
  bounded functional with
  \begin{equation} \label{eq:intro-functionalbound}
    \left\| L_u \right\| \leq \int_0^\infty \left(
      \left|u''(x)\right|^2
      - \hardyconst{1}{2} \frac{|u(x)|^2}{x^4}
      + C^{1/\gamma} |u(x)|^2
    \right)  \, dx,
  \end{equation}
  The Riesz representation theorem give us that $|u|^2 \in \lspace p
  {0,\infty}$ with
  \begin{equation*}
    \left\||u|^2\right\|_p 
    = \|L_u\|
  \end{equation*}
  and since $u \in \cz$ was arbitrary, the result follows from
  \eqref{eq:intro-functionalbound}.
\end{proof}

\section{An Auxiliary Operator on a Finite Interval} \label{sec:interval}

In this section, some auxiliary results for a certain operator on a
finite subinterval of $(0,\infty)$ will be proved. These results are the key
ingredients in the proof of Theorem \ref{thm:maingeneral}.
Throughout this section, the constants $\alpha$ and $\beta$ will be fixed and
satisfy
\begin{equation*}
  \alpha \geq 0
  \quad \text{ and } \quad
  0 \leq \beta < \frac 3 2 + \alpha.
\end{equation*}
For $b > 0$, define the closed quadratic form
\begin{equation*}
  h_b[u] = \int_b^{b+1} \left|
    \left(\frac{1}{x^\alpha}\left(\frac{u(x)}{x^\beta}\right)'\right)'
  \right|^2 x^{2(\alpha+\beta)} \, dx,
\end{equation*}
with domain $\qdom{h_b} = \htwob$.
As usual, there is a canonically defined sesqui-linear form, denoted by
$h_b[\cdot,\cdot]$ from which $h_b$ arises, but we will make no
distinction between the quadratic and sesqui-linear forms.
Let $H_b$ denote the self-adjoint
operator in $\ltwob$ associated with $h_b$. We will henceforth fix the
functions
\begin{equation*}
  f_1(x) = x^{\beta}
  \quad \text{ and } \quad
  f_2(x) = x^{\alpha + \beta + 1}.
\end{equation*}
and note that $f_1, f_2 \in \qdom{h_b}$, as well as $h_b[f_1,v] =
h_b[f_2,v] = 0$ for any function $v \in \qdom{h_b}$. Therefore, $f_1, f_2 \in
\dom{H_b}$ and
\begin{equation*}
  H_bf_1 = 0 
  \quad \text{ and } \quad
  H_bf_2 = 0.
\end{equation*}

When describing properties for $h_b$, certain natural conditions on
the functions $u \in \qdom{h_b}$ will appear. These conditions are
\begin{equation} \label{eq:cond1}
  \int_b^{b+1}u(x)x^\beta \, dx = 0
\end{equation}
and
\begin{equation} \label{eq:cond2}
  \int_b^{b+1}\left(\frac{u(x)}{x^\beta}\right)'x^\alpha \, dx = 0.
\end{equation}

Let us start by looking more closely at the nature of these
conditions:
\begin{lem} \label{lem:conddescr}
  Let $b > 0$. Then the following hold:
  \begin{enumerate}[{\normalfont (i)}]
  \item
    If $u \in \qdom{h_b}$ satisfies $u \neq 0$, \eqref{eq:cond1} and
    \eqref{eq:cond2}, then the functions $f_1$, $f_2$ and $u$ are 
    linearly independent.
  \item \label{it:decompcond}
    For any $v \in \qdom{h_b}$, there are constants
    $c_1, c_2 \in \complex$ and a function $u \in
    \qdom{h_b}$ such that \eqref{eq:cond1} and \eqref{eq:cond2} hold
    and such that
    \begin{equation*}
      v = c_1f_1 + c_2f_2 + u.
    \end{equation*}
  \item \label{it:dimcond}
    Suppose that $F \subset \qdom{h_b}$ in a linear set and has
    $\dim F \geq 3$. Then there is $u \in F$ with $u \neq 0$ that satisfies
    \eqref{eq:cond1} and \eqref{eq:cond2}.
  \end{enumerate}
\end{lem}
\begin{proof}
  It will be useful to consider the transformation $T$ defined by 
  \begin{equation*}
    (Tv)(x) = \left(\frac{v(x)}{x^\beta}\right)',
    \quad v \in \honeb,
  \end{equation*}
  as well as the function $g(x) = x^\alpha$. Note that $Tf_1 =
  0$ and $Tf_2 = (\alpha + 1)g$.
  \begin{enumerate}[{\normalfont (i)}]
  \item 
    Let $u \in \qdom{h_b}$ satisfy $u \neq 0$, \eqref{eq:cond1} and
    \eqref{eq:cond2}. Because of \eqref{eq:cond1}, it must be that $Tu
    \neq 0$. Also, since $\int Tu \cdot g \, dx = 0$ by
    \eqref{eq:cond2}, the functions $Tu$ and $g$ are linearly independent.
    
    Choose scalars $\eta_1, \eta_2, \eta_3$ such that 
    $\eta_1f_1 + \eta_2f_2 + \eta_3u = 0$. Then
    \begin{equation*}
      0 = T\left(\eta_1f_1 + \eta_2f_2 + \eta_3u\right)
      = \eta_2 (\alpha + 1)g + \eta_3Tu.
    \end{equation*}
    But since the functions $g$ and $Tu$ are 
    linearly independent, it must be
    that $\eta_2 = \eta_3 = 0$. It follows that $\eta_1f_1 = 0$, and
    thus also $\eta_1 = 0$. 

    \item 
      Choose $v \in \qdom{h_b}$.
      Let $w = Tv \in \ltwob$. Using orthogonal
      projections one finds that there is a constant
      $c$ and a function $\tilde w$ such that $w = cg + \tilde
      w$ and $(\tilde w, g) = 0$. Let 
      \begin{equation*}
        \tilde u(x) = x^{\beta} \int_b^x \tilde w(t) \, dt
        \quad \text{ and } \quad
        c_2 = \frac{c}{\alpha + 1}.
      \end{equation*}
      Again by projecting, we find a constant $d$ and a 
      function $u$ such that $\tilde u =
      df_1 + u$ and $(u,f_1) = 0$, that is, $u$ satisfies
      $\eqref{eq:cond1}$. 
      Note that $Tu = T\tilde u - dTf_1 = T\tilde u = \tilde w$. 
      Since $(\tilde w, g) = 0$, we get that $u$ satisfies 
      \eqref{eq:cond2}. 
      Now, 
      $T(c_2 f_2 + u) = cg + \tilde w = w$. Hence
      $T\left(v - (c_2f_2 + u)\right) = 0$,
      and therefore there must be a constant $c_1$ such that
      \begin{equation*}
        \frac{v(x) - (c_2f_2(x) + u(x))}{x^\beta} = c_1.
      \end{equation*}
      Thus,
      \begin{equation*}
        v = c_1f_1 + c_2f_2 + u,
      \end{equation*}
      and since $f_1,f_2,v \in \qdom{h_b}$, it follows that $u \in
      \qdom{h_b}$.

    \item Clearly, $F$ has a two-dimensional subspace $G$ orthogonal
      to $\linspan\{f_1\}$. This means that every $u \in G$ satisfies
      \eqref{eq:cond1}. 

      Let us prove that $TG$ is
      two-dimensional. Let $v_1, v_2 \in G$ be linearly
      independent. Choose $c_1, c_2 \in \complex$ such that 
      $c_1Tv_1 + c_2Tv_2 = 0$. It follows ny the definition of $T$ 
      that $c_1v_1 + c_2v_2 = cf_1$ for some constant $c \in
      \complex$. Since the functions $v_1, v_2$ and $f_1$ are linearly
      independent, it must be that $c_1 = c_2 = 0$. This shows that
      $Tv_1$ and $Tv_2$ are linearly independent. Hence $TG$ is
      two-dimensional. 

      In particular there is a $w \in TG$, with $w \neq 0$, orthogonal
      to $g$. Find $u \in G$ such that $Tu = w$. Then $u$ satisfies
      $u \neq 0$ and \eqref{eq:cond2}. And since $u \in G$, it also
      satisfies \eqref{eq:cond1}.
      \qedhere
  \end{enumerate}
\end{proof}

\begin{lem} \label{lem:supu}
  Let $0 \leq \xp < 3$ and $\xp \leq 2\beta$. 
  There is a constant $C_\xp > 0$ such that
  \begin{equation*}
    \sup_{b \leq y \leq b+1}\frac{|u(y)|^2}{y^{\xp}} \leq C_\xp h_b[u],
  \end{equation*}
  for any $b > 0$ and any
  $u \in \qdom{h_b}$ that satisfies \eqref{eq:cond1} and \eqref{eq:cond2}.
\end{lem}
\begin{proof}
  The proof is divided into two cases. The first case is when $0 < b
  \leq 1$ and the second case is when $b > 1$.

  \emph{Case 1}. Choose $b$ with $0 < b \leq 1$ and $u \in \qdom{h_b}$
  with $u \neq 0$, that satisfies \eqref{eq:cond1} and \eqref{eq:cond2}. 
  Also choose $y$ with $b \leq y \leq b+1$. Clearly, 
  \begin{align*}
    \frac{u(y)}{y^\beta}\left(
      y^{2\beta+1}\!\!\right.&\left.\,-b^{2\beta+1}
    \right)
    = \int_b^y \!\! \left(
      \frac{u(t)}{t^\beta}\left(t^{2\beta+1}\!\!-\!\!b^{2\beta+1}\right)
    \right)' \, dt \\
    & = (2\beta\!\!+\!\!1)\int_b^y \!\! u(t)t^\beta \, dt
    + \int_b^y \!\! \left(\frac{u(t)}{t^\beta}\right)' \!\!
    \left(t^{2\beta+1} \!\! - \!\! b^{2\beta+1}\right) \, dt,
  \end{align*}
  and also
  \begin{align*}
    \frac{u(y)}{y^\beta}\left(
      (b\!\!+\!\!1)^{2\beta+1}\!\!\right.&\left.\,-y^{2\beta+1}
    \right)
    = -\int_y^{b+1} \!\! \left(
      \frac{u(t)}{t^\beta}\left(
        (b\!\!+\!\!1)^{2\beta+1}\!\!-\!\!t^{2\beta+1}\right)
    \right)' \, dt \\
    & = (2\beta\!\!+\!\!1)\int_y^{b+1} \!\!\!\!\!\!
    u(t)t^\beta \, dt
    - \int_y^{b+1}\!\!\left(\frac{u(t)}{t^\beta}\right)' \!\!
    \left(
      (b\!\!+\!\!1)^{2\beta+1} \!\! - \!\! t^{2\beta+1}
    \right) \, dt.
  \end{align*}
  Hence, using \eqref{eq:cond1}, 
  \begin{align} \label{eq:firstder}
    \frac{u(y)}{y^\beta}
    =
    \Gamma_1 
    & \left( \int_b^y \left(\frac{u(t)}{t^\beta}\right)'  \left(
        t^{2\beta+1} - b^{2\beta+1}
      \right) \, dt \right. \\
    & \left. -
      \int_y^{b+1}\left(\frac{u(t)}{t^\beta}\right)' 
      \left(
        (b+1)^{2\beta+1} - t^{2\beta+1}
      \right) \, dt 
    \right)
  \end{align}
  where
  \begin{equation*}
    \Gamma_1 = \frac{1}{(b + 1)^{2\beta+1} - b^{2\beta+1}}.
  \end{equation*}
  Similarly to the above, using \eqref{eq:cond2}, we find
  that for any $t$ with $b \leq t \leq b+1$, 
  \begin{align} \label{eq:secondder}
    \frac{1}{t^\alpha}\left(
      \frac{u(t)}{t^\beta}
    \right)'
    = 
    \Gamma_2
    & \left(\int_b^t \left(
        \frac{1}{x^\alpha}\left(\frac{u(x)}{x^\beta}\right)'
      \right)'\left(x^{2\alpha+1}-b^{2\alpha+1}\right) \, dx
    \right. \\
    & \left. - \int_t^{b+1}
      \left(
        \frac{1}{x^\alpha}\left(\frac{u(x)}{x^\beta}\right)'
      \right)'
      \left((b+1)^{2\beta+1}-x^{2\beta+1}\right)
      \, dx
    \right),
  \end{align}
  where
  \begin{equation*}
    \Gamma_2 = \frac{1}{(b + 1)^{2\alpha+1} - b^{2\alpha+1}}.
  \end{equation*}
  Combine \eqref{eq:firstder} and \eqref{eq:secondder} to obtain
  \begin{equation*}
    \frac{u(y)}{y^{\xp/2}} = \Gamma_1\Gamma_2 \cdot y^{\beta -
      \frac{\xp}{2}} \Big( 
      I_1 - I_2 - I_3 + I_4
    \Big),
  \end{equation*}
  where
  \begin{equation*}
    I_1 =
    \int_b^y
    \!\!\!\! \int_b^t \!\!
    \left(
          \frac{1}{x^\alpha}\left(\frac{u(x)}{x^\beta}\right)'
        \right)' \!\!
        \left(x^{2\alpha+1} \! - \! b^{2\alpha+1}\right)
    t^\alpha 
    \left(t^{2\beta+1} \! - \! b^{2\beta+1}\right) \, dx \, dt,
  \end{equation*} 
  \begin{equation*}
    I_2 =
    \int_y^{b+1}
    \!\!\!\! \int_b^t \!\!
    \left(
          \frac{1}{x^\alpha}\left(\frac{u(x)}{x^\beta}\right)'
        \right)' \!\!
        \left(x^{2\alpha+1} \! - \! b^{2\alpha+1}\right)
    t^\alpha
    \left(
      (b\!\!+\!\!1)^{2\beta+1} \! - \! t^{2\beta+1}
    \right) \, dx \, dt ,
  \end{equation*}
  \begin{equation*}
    I_3 =
    \int_b^y
    \!\!\!\! \int_t^{b+1} \!\!
    \left(
      \frac{1}{x^\alpha}\left(\frac{u(x)}{x^\beta}\right)'
    \right)' \!\!
    \left((b\!\!+\!\!1)^{2\alpha+1} \! - \! x^{2\alpha+1}\right)
    t^\alpha 
    \left(t^{2\beta+1} \! - \! b^{2\beta+1}\right) \, dx \, dt
  \end{equation*}
  and
  \begin{equation*}
    I_4 =
    \int_y^{b+1}
    \!\!\!\! \int_t^{b+1} \!\!
    \left(
      \frac{1}{x^\alpha}\left(\frac{u(x)}{x^\beta}\right)'
    \right)' \!\!
    \left((b\!\!+\!\!1)^{2\alpha+1} \! - \! x^{2\alpha+1}\right)
    t^\alpha
    \left(
      (b\!\!+\!\!1)^{2\beta+1} \! - \! t^{2\beta+1}
    \right) \, dx \, dt.
  \end{equation*}
  It remains to find constants $C_1, C_2, C_3, C_4 > 0$, independent of
  $b, u$ and $y$, such that
  \begin{equation*}
    \left|y^{\beta - \frac{\xp}{2}} I_j\right|^2 \leq C_j h_b[u],
    \quad j = 1,2,3,4,
  \end{equation*}
  since then we can use the fact that $\Gamma_1 \leq 1$ and 
  $\Gamma_2 \leq 1$, to conclude that
  \begin{equation*}
    \frac{|u(y)|^2}{y^\xp} \leq 4\left(C_1 + C_2 + C_3 + C_4\right) h_b[u].
  \end{equation*}
  Note that by the Cauchy-Schwarz inequality,
  \begin{equation*}
    \left|y^{\beta - \frac{\xp}{2}} I_1\right|^2
    \leq 
    J_1(b,y)^2 \cdot h_b[u]
  \end{equation*}
  where
  \begin{equation*}
    J(b,y) = 
    y^{\beta - \frac{\xp}{2}} \int_b^y t^{\alpha}
    \left(t^{2\beta+1}-b^{2\beta+1}\right)
    \left(\int_b^t
      \frac{\left(x^{2\alpha+1}-b^{2\alpha+1}\right)^2}{x^{2(\alpha+\beta)}}
      \, dx
    \right)^{1/2} \!\! dt,
  \end{equation*}
  and similarly for $I_2$, $I_3$ and $I_4$. 
  Now, as soon as $\alpha \geq 0$, $0 \leq \beta < 3/2 + \alpha$, 
  $0 \leq \xp < 3$ and $\xp \leq 2\beta$, some calculations
  show that there exists constants $C_1, C_2, C_3, C_4 > 0$ such that 
  for any $b > 0$ and $y$ with $b \leq y \leq b+1$
  \begin{equation*}
    J_n(b,y)^2 \leq C_n, \quad n = 1,2,3,4.
  \end{equation*}

  \emph{Case 2}. Choose $b > 1$, $u \in \qdom{h_b}$ that satisfies 
  \eqref{eq:cond1} and \eqref{eq:cond2}, and $y \in
  [b,b+1]$. Set $v(x) = u(x) / x^\beta$. By \eqref{eq:cond1}
  and \eqref{eq:cond2}, there are points $x_1,x_2 \in
  [b,b+1]$ such that
  \begin{equation*}
    v(x_1) = 0
    \quad \text{ and } \quad
    v'(x_2) = 0.
  \end{equation*}
  Let us start by showing that 
  \begin{equation} \label{eq:vprime}
    \left| \frac{v'(t)}{t^\alpha} \right|^2
    \leq h_b[u] \frac{1}{b^{2(\alpha+\beta)}},
    \quad b \leq t \leq b+1.
  \end{equation}
  Choose $t$ with $b \leq t \leq b+1$. Assume first that $b \geq
  x_2$. 
  Note that
  \begin{equation*}
    \frac{v'(t)}{t^\alpha} = \int_{x_2}^t \left( 
      \frac{v'(x)}{x^\alpha}
    \right)' x^{\alpha+\beta} \cdot \frac{1}{x^{\alpha+\beta}} \, dx.
  \end{equation*}
  Hence,
  \begin{equation*}
    \left| \frac{v'(t)}{t^\alpha} \right|^2
    \leq \int_{x_2}^t \left|
      \left(\frac{v'(t)}{x^\alpha}\right)'x^{\alpha+\beta}
    \right|^2 dx 
    \cdot \int_{x_2}^t \frac{1}{x^{2(\alpha+\beta)}} \, dx
    \leq h_b[u] \cdot \frac{1}{b^{2(\alpha+\beta)}}
  \end{equation*}
  The case when $b \leq x_2$ is handled in a similar way. This
  concludes the proof of equation \eqref{eq:vprime}.

  Now, assume that $y \geq x_1$. It follows from \eqref{eq:vprime} that 
  \begin{align*}
    \left|v(y)\right|^2
    & = \left| \int_{x_1}^y \frac{v'(t)}{t^\alpha} \cdot t^\alpha \, dt
    \right|^2 \\
    & \leq \int_{x_1}^y \left|\frac{v'(t)}{t^\alpha}\right|^2 \, dt
    \cdot \int_{x_1}^y t^{2\alpha} \, dt \\
    & \leq h_b[u] \cdot \frac{1}{b^{2(\alpha+\beta)}} 
    \cdot (b+1)^{2\alpha}.
  \end{align*}
  Therefore, since $b > 1$
  \begin{equation*}
    \frac{u(y)}{y^{\xp}} 
    \leq \frac{1}{y^{\xp}} 
    \cdot \frac{(b+1)^{2(\alpha+\beta)}}{b^{2(\alpha+\beta)}}
    \cdot h_b[u]
    \leq 2^{2(\alpha+\beta)} \cdot h_b[u].
  \end{equation*}
  By similar arguments, this inequality also holds when $y \leq x_1$.
\end{proof}

\begin{prop} \label{prop:interval}
  Let $0 \leq \xp < 3$ and $\xp \leq 2\beta$.
  There are constants $D_\xp, E_\xp > 0$ such that for any $b > 0$ and any
  non-negative bounded potential $V \neq 0$ that satisfies
  \begin{equation} \label{eq:intV}
    \int_b^{b+1} V(x) x^{\xp} \, dx \leq D_\xp,
  \end{equation}
  the operator $H_b - V$ has negative spectrum consisting of 
  exactly two eigenvalues, $-E_1$ and $-E_2$, that satisfy
  \begin{equation*}
    E_1 \leq E_\xp \quad \text{ and } \quad
    E_2 \leq E_\xp.
  \end{equation*}
\end{prop}
\begin{proof}
  For given $b > 0$, let
  \begin{equation*}
    g_b = f_2 - \frac{\int_b^{b+1}f_1 f_2 \, dx}{\int_b^{b+1}
      |f_1|^2 \, dx} \cdot f_1,
  \end{equation*}
  so that $f_1$ and $g_b$ are orthogonal in $\ltwob$ and 
  span the same subspace as $f_1$ and $f_2$.
  Introduce 
  \begin{equation*}
    B_1 = \frac 1 4 \inf_{b > 0} \left(
      \frac{\int_b^{b+1} |f_1(x)|^2 \, dx}
      {\sup\limits_{b \leq y \leq b+1} \frac{|f_1(y)|^2}{y^{\xp}}}
    \right)
    \quad \text{ and } \quad
    B_2 = \frac 1 4 \inf_{b > 0} \left(
      \frac{\int_b^{b+1} |g_b(x)|^2 \, dx}
      {\sup\limits_{b \leq y \leq b+1} \frac{|g_b(y)|^2}{y^{\xp}}}
    \right).
  \end{equation*}
  Elementary calculations show that $B_1 > 0$ and 
  $B_2 > 0$. 
  Now set
  \begin{equation*}
    B = \min\left\{B_1, B_2 \right\}.
  \end{equation*}
  Let $C_\xp > 0$ be as in Lemma \ref{lem:supu}, and
  \begin{equation*}
    D_\xp = \min\left\{
      \frac{B}{2C_\xp(1 + B)},
      \frac{B}{2C_0(1 + B)},
      \frac{1}{C_\xp}
    \right\}
    \quad \text{ and } \quad
    E_\xp = \frac{1}{C_0(1+B)}.
  \end{equation*}
  Note that this in particular implies
  \begin{equation} \label{eq:CDleq}
    C_\xp D_\xp \leq 1,
  \end{equation}
  \begin{equation} \label{eq:DpEleq}
    C_0 E_\xp + 2C_\xp D_\xp \leq 1
  \end{equation}
  and
  \begin{equation} \label{eq:BEgeq2D}
    B E_\xp \geq 2 D_\xp.
  \end{equation}
  Choose $b > 0$ and a non-negative $V \neq 0$ 
  that satisfies condition \eqref{eq:intV}.

  \emph{Step 1}.
  Let us start by showing that $H_b - V \geq -E_\xp$, in the sense of
  quadratic forms. Choose $v \in \qdom{h_b}$. By (\ref{it:decompcond}) 
  of Lemma \ref{lem:conddescr} there are constants $c_1,c_2 \in \complex$ 
  and a function $u \in \qdom{h_b}$ for which \eqref{eq:cond1} and
  \eqref{eq:cond2} hold, such that
  \begin{equation*}
    v = c_1f_1 + c_2 g_b + u.
  \end{equation*}
  Write
  $w = c_1f_1 + c_2 g_b$ and
  observe that 
  \begin{align*}
    \frac{E_\xp}{2} \int_b^{b+1} \!\!\!\! \left|w(x)\right|^2 \, dx 
    & = \frac{E_\xp}{2} \left(
      |c_1|^2 \int_b^{b+1} \!\!\!\! \left|f_1(x)\right|^2 \, dx
      + |c_2|^2 \int_b^{b+1} \!\!\!\! \left|g_b(x)\right|^2 \, dx
    \right) \\
    & \geq E_\xp \left(
      B_1 \cdot 2 |c_1|^2 \!\! 
      \sup_{b \leq y \leq b+1} \!\!\!\!\! \frac{|f_1(y)|^2}{y^{\xp}}
      + B_2 \cdot 2 |c_2|^2 \!\! 
      \sup_{b \leq y \leq b+1} \!\!\!\!\! \frac{|g_b(y)|^2}{y^{\xp}}
    \right) \\
    & \geq B E_\xp \cdot
      \sup_{b \leq y \leq b+1} \!\! \frac{|w(y)|^2 }{y^{\xp}}.
  \end{align*}
  In other words, using \eqref{eq:BEgeq2D},
  \begin{equation} \label{eq:intw}
    \frac{E_\xp}{2} \int_b^{b+1} \!\!\!\! \left|w(x)\right|^2 \, dx 
    - 2 D_\xp \!\! \sup_{b \leq y \leq b+1} \!\! 
    \frac{\left|w(y)\right|^2}{y^{\xp}} \geq 0.
  \end{equation}

  Consider the quadratic expression
  \begin{equation*}
    g[v] = h_b[v] - \int_b^{b+1} \!\!\!\! V(x) |v(x)|^2 \, dx 
    + E_\xp \int_b^{b+1}|v(x)|^2 \, dx.
  \end{equation*}
  We have to prove that $g[v] \geq 0$. Note that $h_b[v] = h_b[u]$.
  Then use the
  facts that $|a + b|^2 \leq 2|a|^2 + 2|b|^2$ and 
  $|a + b|^2 \geq \frac{1}{2}|a|^2 - |b|^2$, for any $a,b \in
  \complex$, together with
  \eqref{eq:intV}, \eqref{eq:intw} and Lemma \ref{lem:supu}, to get that
  \begin{align*}
    g[v]
    & \geq h_b[u] - \int_b^{b+1} \!\!\!\! 
    \left(E_\xp + 2V(x)\right) |u(x)|^2 \, dx
    + \int_b^{b+1} \!\! \left(\frac{E_\xp}{2} - 2V(x)\right)
    \left|w(x)\right|^2 \, dx \\
    & \geq h_b[u] - 
    \!\! \sup_{b \leq y \leq b+1} \!\! |u(y)|^2 \cdot E_\xp
    -2 \!\! \sup_{b \leq y \leq b+1} \!\! \frac{|u(y)|^2}{y^\xp} 
    \cdot D_\xp \\
    & \phantom{\geq \,}
    + \frac{E_\xp}{2} \int_b^{b+1} \!\!\!\! \left|w(x)\right|^2 \, dx 
    - 2 D_\xp \!\! \sup_{b \leq y \leq b+1} \!\! 
    \frac{\left|w(y)\right|^2}{y^{\xp}}
    \\
    & \geq h_b[u] \left(1 - (C_0 E_\xp + 2C_\xp D_\xp)\right)
    \geq 0.
  \end{align*}
  
  \emph{Step 2}. Continue by showing that the negative spectrum of
  $H_b - V$ is discrete and consists of at most two
  eigenvalues. Let $E$ be the spectral measure corresponding
  to $H_b - V$, and 
  denote by $N_-(H_b - V)$ the rank of $E(-\infty,0)$.
  From Glazman's lemma, see Remark \ref{rem:glazman}, it is known that
  \begin{equation} \label{eq:glazman}
    N_-(H_b - V) = \sup_F \, \dim F,
  \end{equation}
  where the supremum is taken over all linear subsets $F \subset
  \qdom{h_b}$ such that
  \begin{equation*}
    h_b[f] - \int_b^{b+1}V|f|^2 \, dx < 0, 
  \end{equation*}
  for any $f \in F$ with $f \neq 0$. Suppose that $F \subset
  \qdom{h_b}$ is a linear set that satisfies $\dim F \geq 3$. By 
  (\ref{it:dimcond}) of Lemma \ref{lem:conddescr} we find a $u \in F$ 
  with $u \neq 0$
  that satisfies \eqref{eq:cond1} and \eqref{eq:cond2}. 
  It follows from Lemma \ref{lem:supu}, \eqref{eq:intV}
  and \eqref{eq:CDleq} that
  \begin{align*}
    h_b[u] - \int_b^{b+1}V(x)|u(x)|^2 \, dx
    & \geq h_b[u] - \sup_{b \leq y \leq b+1}\frac{|u(y)|^2}{y^{\xp}}
    \int_b^{b+1} V(x) x^{\xp} \, dx \\
    & \geq h_b[u] - C_\xp D_\xp h_b[u] \\
    & \geq 0.
  \end{align*}
  Using \eqref{eq:glazman}, this shows that $N_-(H_b - V) \leq 2$.

  \emph{Step 3}. The final step is to provide the existence of at
  least two negative eigenvalues of $H_b - V$. Let $F =
  \linspan\left\{f_1,f_2\right\}$. Clearly $h_b[f] = 0$ for any $f \in
  F$. In particular
  \begin{equation*}
    h_b[f] - \int_b^{b+1} V|f|^2 \, dx < 0, 
  \end{equation*}
  for $f \in F$ with $f \neq 0$.
  Again using \eqref{eq:glazman}, it is seen that $N_-(H_b - V) \geq 2$.
\end{proof}

\begin{rem}[Glazman's lemma] \label{rem:glazman}
  Much of the variational techniques used here origin in
  \emph{Glazman's lemma}. This lemma, discussed in e.g.
  \cite{birman-solomyak-paper} states the following:

  Let $\hilbert$ be a separable Hilbert space, and $D$ a dense, linear
  subset of $\hilbert$. Suppose that the semi-bounded quadratic form $a$ is
  closable on $D$, and denote by $A$ the self-adjoint operator
  corresponding to the closure of $a$. Also let $E$ be the spectral
  measure corresponding to $A$. Then
  \begin{equation*}
    \rank E(-\infty, x) = \sup \dim F,
  \end{equation*}
  where the supremum if taken over all linear subsets $F \subset D$
  such that
  \begin{equation*}
    a[f] < x \|f\|^2,
    \quad f \in F, f \neq 0.
  \end{equation*}
\end{rem}

\section{Well-Behaved Potentials}
In this section we restrict ourselves to bounded potentials $V$ of compact
support. One obvious reason for this is to avoid difficulties in
defining operators
\begin{equation*}
 H = H_0 - V, 
 \quad \text{ where } \quad
 H_0 = (-\Delta)^2 - \hardyconst{d}{2}\frac{1}{|x|^4}.
\end{equation*}
Indeed, when $V$ is bounded, this operator is 
simply defined as an operator sum with domain
$\dom{H} = \dom{H_0}$.

\subsection{General Half-Line Case}
We start by proving the general half-line result that follows from 
Proposition \ref{prop:interval}. The proof uses techniques from
\cite{weidl} and \cite{efjems}.

\begin{prop} \label{prop:niceV-general}
  Let $a \geq 0$, $\beta \geq 0$ and $0 \leq \xp < 3$, $\xp \leq 2\beta$. 
  Define the quadratic form $h_0$ as the closure of
  \begin{equation*}
    u \mapsto
    \int_0^\infty \left|
      \frac{d}{dx} \left(
        \frac{1}{x^\alpha}
        \frac{d}{dx} \left(\frac{u(x)}{x^\beta}\right)
        \right)
    \right|^2 x^{2(\alpha+\beta)} \, dx,
  \end{equation*}
  on $\cz$.
  Let $H_0$ be the self-adjoint operator in $\ltwo$
  corresponding to $h_0$.
  Then there is a
  constant $C = C(\alpha,\beta,\xp) > 0$ such that 
  for any non-negative bounded potential $V$ with compact support 
  in $(0,\infty)$, the negative spectrum of $H_0-V$ is discrete and
  \begin{equation*}
    \tr (H_0 - V)_-^{\frac{3-\xp}{4}} \leq C \int_0^\infty V(x)
    x^{\xp} \, dx.
  \end{equation*}
\end{prop}
\begin{proof}
  Let $D_\xp$ and $E_\xp$ be as in Proposition \ref{prop:interval}. 
  Choose a bounded $V \geq 0$ with $\supp V \Subset (0,\infty)$.
  Define a sequence of numbers $a_1 < a_2 < \cdots$ by
  setting $a_1 = \min \, \supp V > 0$ and recursively,
  \begin{equation} \label{eq:ajdef}
    (a_{j+1} - a_j)^{3-\xp} \int_{a_j}^{a_{j+1}} V(x) x^{\xp} \, dx
    = D_\xp,
  \end{equation}
  for $j \geq 2$. The recursion stops for $j = n$, when 
  $a_n \geq \max \, \supp V$. 
  Indeed, the sequence is finite, since 
  \begin{equation*}
    a_{j+1} - a_j \geq \left(\frac{D_\xp}
      {\int_{0}^{\infty} V(x) x^{\xp} \, dx}\right)^{\frac{1}{3-\xp}}
  \end{equation*}
  for any $j$. Let $a_0 = 0$ and $a_{n+1} = \infty$.

  For $0 \leq j \leq n$, consider the quadratic forms 
  \begin{equation*}
    g_j[v] = \int_{a_j}^{a_{j+1}} \left|
      \frac{d}{dx} \left( \frac{1}{x^\alpha} \frac{d}{dx} 
        \left(\frac{v(x)}{x^\beta} \right)
      \right) \right|^2 x^{2(\alpha+\beta)} \, dx.
  \end{equation*}
  For $1 \leq j \leq n$, the domain of $g_j$ is $\qdom{g_j} =
  \sobspace 2 {a_j,a_{j+1}}$. For $j = 0$, it can be shown that $g_0$
  is closable on $\czeroinf{(0,a_1]}$, and we consider it as the
  closure on this domain.
  In this way all the forms $g_j$ are closed.
  Let $G_j$ be the self-adjoint operator in $\lspace 2 {a_j, a_{j+1}}$
  corresponding to $g_j$.

  By comparing quadratic forms and their domains of
  definition, we get that
  \begin{equation*}
    H_0-V \geq \bigoplus_{j=0}^{n} (G_j - V).
  \end{equation*}
  Note that $G_j - V \geq 0$ for $j = 0$ and for $j = n$, since $V = 0$ on
  $(a_j,a_{j+1})$ for $j=0,n$. Therefore, if we can prove that the
  negative spectrum of $G_j - V$ is discrete for $j=1,2,\ldots,n-1$,
  it follows that the negative spectrum of $H_0-V$ is also discrete. In
  this case,
  \begin{equation} \label{eq:trleq}
    \tr (H_0-V)_-^{\frac{3-\xp}{4}} 
    \leq \sum_{j=1}^{n-1} \tr (G_j-V)_-^{\frac{3-\xp}{4}}.
  \end{equation}

  Let $1 \leq j \leq n-1$ and define a unitary transformation
  $\utrans$ from $\lspace 2 {a_j,a_{j+1}}$ onto $\lspace 2 {b_j,b_j + 1}$ by
  \begin{equation*}
    (\utrans v)(x) = (a_{j+1}-a_j)^{1/2} v\left((a_{j+1}-a_j)x\right),
  \end{equation*}
  where $b_j = a_j / (a_{j+1} - a_j)$. Clearly,
  \begin{equation*}
    G_j - V = \frac{1}{\left(a_{j+1}-a_j\right)^4} 
    \utrans^{-1}\left(H_{b_j} - V_j\right)\utrans,
  \end{equation*}
  where $H_b$ is the operator discussed in Section \ref{sec:interval}
  and where
  \begin{equation*}
    V_j(x) = (a_{j+1}-a_j)^4 V\left((a_{j+1}-a_j)x \right).
  \end{equation*}
  Note that by \eqref{eq:ajdef},
  \begin{equation*}
    \int_{b_j}^{b_j+1} V_j(x) x^{\xp} \, dx = D_\xp.
  \end{equation*}
  Hence by Proposition \ref{prop:interval}, the negative spectrum of 
  $H_{b_j} - V_j$, an therefore also of $G_j - V$, is discrete. 
  By the same proposition and \eqref{eq:ajdef},
  \begin{align*}
    \tr (G_j - V)_-^{\frac{3-\xp}{4}}
    & = \frac{1}{(a_{j+1}-a_j)^{3-\xp}} 
    \tr \left(H_{b_j} - V_j\right)_-^{\frac{3-\xp}{4}} \\
    & \leq \frac{2}{(a_{j+1}-a_j)^{3-\xp}} E_\xp^{\frac{3-\xp}{4}} \\
    & = \frac{2E_\xp^{\frac{3-\xp}{4}}}{D_\xp} 
    \int_{a_j}^{a_{j+1}} V(x) x^{\xp} \, dx.
  \end{align*}
  We have thus shown that the negative spectrum of $H_0-V$ is
  discrete, and by \eqref{eq:trleq}, that
  \begin{equation*}
    \tr(H_0-V)_-^{\frac{3-\xp}{4}} 
    \leq \frac{2E_\xp^{\frac{3-\xp}{4}}}{D_\xp} 
    \int_0^\infty V(x) x^{\xp} \, dx.
    \qedhere
  \end{equation*}
\end{proof}

By standard methods, as described in \cite{hundertmark}, and originally in
\cite{aizenman-lieb}, this result extends to an inequality for
$\tr(H_0-V)_-^\gamma$ for any $\gamma \geq \gc$, where
\begin{equation*}
  \gc = \frac{3 - \xp}{4}.
\end{equation*}

\begin{cor} \label{cor:niceV-general}
  Let $\alpha, \beta, \xp$ and $H_0$ be as in Proposition
  \ref{prop:niceV-general}. Then, for
  any $\gamma \geq \gc$, there is a 
  $C = C(\alpha, \beta, \xp, \gamma) > 0$ such that for any
  non-negative, bounded potential $V$ with compact support in
  $(0,\infty)$,
  \begin{equation*}
    \tr(H_0-V)_-^\gamma \leq C \int_0^\infty V(x)^{1+\gamma-\gc}
    x^{\xp} \, dx.
  \end{equation*}
\end{cor}

\subsection{The Fourth-Order Operator on the Half-Line}
We now consider the fourth-order operator 
$d^4/dx^4 - \hardyconst{1}{2} / x^4 - V(x)$ 
in $\ltwo$, where $\hardyconst{1}{2} = 9/16$ is the sharp constant in the
classical Hardy-Rellich inequality
\begin{equation*}
  \int_0^\infty |u''(x)|^2 \, dx 
  \geq \hardyconst{1}{2} \int_0^\infty \frac{|u(x)|^2}{x^4} \, dx,
  \quad u \in \cz.
\end{equation*}

\begin{prop} \label{prop:niceV-1d}
  Let $0 \leq \xp < 3$ and consider the operator
  \begin{equation*}
    H_0 = \frac{d^4}{dx^4} - \hardyconst{1}{2} \frac{1}{x^4}
  \end{equation*}
  in $\ltwo$
  defined as the Friedrich extension of the corresponding operator
  initially defined on $\cz$.   
  Then, for any $\gamma \geq (3-\xp)/4$, there is a 
  $C = C(\xp, \gamma) > 0$ such that for any
  non-negative, bounded potential $V$ with compact support in
  $(0,\infty)$, the negative spectrum of $H_0 - V$ is discrete and
  \begin{equation*}
    \tr(H_0-V)_-^\gamma \leq C \int_0^\infty V(x)^{\gamma+\frac{1+\xp}{4}}
    x^{\xp} \, dx.
  \end{equation*}
\end{prop}
\begin{proof}
  Note that since $\hardyconst{1}{2} = 9/16$, 
  the closed quadratic form $h_0$ corresponding to $H_0$ is the
  closure of 
  \begin{equation*}
    u \mapsto \int_0^\infty \left(
      |u''(x)|^2 - \frac{9}{16} \cdot \frac{|u(x)|^2}{x^4} 
    \right) \, dx
  \end{equation*}
  on $\cz$. By partial integration we see that for $u \in \cz$,
  \begin{equation*}
    h_0[u] = \int_0^\infty \left|
      \frac{d}{dx} \left(
        \frac{1}{x^\alpha}
        \frac{d}{dx} \left(\frac{u(x)}{x^\beta}\right)
        \right)
    \right|^2 x^{2(\alpha+\beta)} \, dx,
  \end{equation*}
  where
  \begin{equation*}
    \alpha = \frac{\sqrt{10} - 2}{2}
    \quad \text{ and } \quad
    \beta = \frac{3}{2}.
  \end{equation*}
  Hence the result follows from Corollary \ref{cor:niceV-general}.
\end{proof}

\subsection{The Fourth-Order Operator in Three Dimensions}
Let us turn to the operator $(-\Delta)^2 - \hardyconst{3}{2}/|x|^4 -
V(x)$ in $\lspace 2 {\reals^3}$. In $\rd 3$ we have the Hardy-Rellich
inequality
\begin{equation*}
  \int_{\rd 3} |\Delta u(x)|^2 \, dx 
  \geq \hardyconst{3}{2} \int_{\rd 3} \frac{|u(x)|^2}{|x|^4} \, dx,
  \quad u \in \czthree,
\end{equation*}
where the sharp constant $\hardyconst{3}{2}$ conveniently enough coincides
with $\hardyconst{1}{2} = 9/16$.

We denote by $\sphere$ and $\sigma$ the unit sphere and
two-dimensional surface measure in $\rd 3$, respectively.
Let $Y_n$, $n=0,1,2,\ldots$ be the normalized eigenfunctions of the
Laplace-Beltrami operator in $\lspace 2{\sphere,\sigma}$. The
eigenfunction $Y_n$ corresponds to the eigenvalue $n(n+1)$, and in
particular $Y_0$ is constant.
Consider the canonical isometric isomorphism 
\begin{equation*}
  \utrans : \lspace 2 {\rd 3} \to \bigoplus_{n=0}^\infty \ltwo
\end{equation*}
given by
\begin{equation*}
  \utrans u = \left\{u_n\right\}_{n=0}^\infty,
  \quad \text{ where } \quad
  u_n(r) = r \int_{\sphere} u(r\theta) Y_n(\theta) \, d\sigma(\theta)
\end{equation*}
for $u \in \czthree$. For such $u$, it is the case that $u_n \in \cz$ and 
\begin{equation} \label{eq:laplacedecomp}
  \utrans(-\Delta u) 
  = \left\{-u_n''(r) + n(n+1)\frac{u_n(r)}{r^2}\right\}_{n=0}^\infty.
\end{equation}

Finally, let $\hilbert = \bigoplus_{n=0}^\infty \ltwo$ and consider
the orthogonal decomposition $\hilbert = \hilbert_1 \oplus \hilbert_2$,
where
\begin{equation*}
  \hilbert_1 = \left\{
    \{u_n\}_{n=0}^\infty \st u_n = 0 \text{ for } n \geq 1
  \right\}
  \quad \text{ and } \quad
  \hilbert_2 = \left\{
    \{u_n\}_{n=0}^\infty \st u_0 = 0
  \right\}.
\end{equation*}
Let $P_1$ and $P_2$ be the orthogonal projections in $\hilbert$ onto
the subspaces $\hilbert_1$ and $\hilbert_2$, respectively. Depending
on context, we will sometimes identify $\hilbert_1$ with the space $\ltwo$.

\begin{lem} \label{lem:orthhardy}
  For any $u \in \cz$ and any $c > 0$,
  \begin{equation*}
    \int_0^\infty \left| -u''(x) + c\frac{u(x)}{x^2}\right|^2 \, dx
    \geq \left(c^2 - \frac 3 2 c + \hardyconst{1}{2}\right)
    \int_0^\infty \frac{|u(x)|^2}{x^4} \, dx.
  \end{equation*}
\end{lem}
\begin{proof}
  Choose $u \in \cz$ and any $c > 0$.
  Recall the classical Hardy-Rellich inequalities
  \begin{equation} \label{eq:hardy1d}
    \int_0^\infty |u'(x)|^2 \, dx 
    \geq \frac 1 4 \int_0^\infty \frac{|u(x)|^2}{x^2} \, dx
  \end{equation}
  and
  \begin{equation} \label{eq:hardy3d}
    \int_0^\infty |u''(x)|^2 \, dx
    \geq \hardyconst{1}{2} \int_0^\infty \frac{|u(x)|^2}{x^4} \, dx.
  \end{equation}
  It follows from \eqref{eq:hardy1d} that
  \begin{equation} \label{eq:frachardy}
    \int_0^\infty \frac{|u'(x)|^2}{x^2} \, dx
    \geq \frac{9}{4} \int_0^\infty \frac{|u(x)|^2}{x^4} \, dx.
  \end{equation}
  Now, by partial integration,
  \begin{equation*}
    \int_0^\infty \left| -u''(x) + c\frac{u(x)}{x^2}\right|^2 \, dx
    = \int_0^\infty \!\! \left(
      |u''(x)|^2 + 2c \frac{|u'(x)|^2}{x^2} 
      + (c^2 \!\! - \!\! 6c) \frac{|u(x)|^2}{x^4}
    \right) \, dx
  \end{equation*}
  Combine this with \eqref{eq:hardy3d} and \eqref{eq:frachardy} to
  obtain the result.
\end{proof}

\begin{lem} \label{lem:3d-firstpart}
  Let $\gamma \geq 1/4$ and let
  $G_0^{(1)}$ be the self-adjoint operator in $\hilbert_1 \cong
  \ltwo$ that corresponds to the closure $g_0^{(1)}$ of the quadratic form
  \begin{equation*}
    u \mapsto \int_0^\infty \left(
      \left|u''(r)\right|^2
      - \hardyconst{1}{2} \frac{|u(r)|^2}{r^4} 
    \right) \, dr,
  \end{equation*}
  initially defined on $\cz$. 
  Then there is a constant $C^{(1)} = C^{(1)}(\gamma) > 0$
  such that given a non-negative $V\in \czthree$ and $V^{(1)} = P_1 \utrans V
  \utrans^{-1} P_1$, the negative spectrum of $G_0 - V^{(1)}$ is
  discrete and 
  \begin{equation*}
    \tr \left( G_0^{(1)} - V^{(1)} \right)_-^{\gamma}
    \leq C^{(1)} \int_{\rd 3} V(x)^{\gamma + \frac{3}{4}} \, dx.
  \end{equation*}
\end{lem}
\begin{proof}
  Use Proposition \ref{prop:niceV-1d} with $\xp = 2$ to obtain a
  constant $C = C(\gamma) > 0$ such that for 
  any non-negative $W \in \cz$, the negative
  spectrum of $G_0^{(1)} - W$ is discrete and
  \begin{equation*}
    \tr\left(G_0^{(1)} - W\right)_-^\gamma \leq
    C \int_0^\infty W(r)^{\gamma + \frac 3 4} \, r^2 \, dr.
  \end{equation*}
  Note that for any non-negative $V \in \czthree$, the operator $V^{(1)}$ is 
  simply multiplication with the function $\tilde V \in \cz$, where
  \begin{equation*}
    \tilde V(r) = \frac{1}{\sigma(\sphere)} 
    \int_{\sphere} V(r\theta) \, d\sigma(\theta),
    \quad r > 0.
  \end{equation*}
  In particular, the negative spectrum of $G_0^{(1)} - V^{(1)}$ is
  discrete and 
  \begin{align*}
    \tr\left(G_0^{(1)} - V^{(1)}\right)_-^\gamma 
    & C \int_0^\infty \tilde V (r)^{\gamma + \frac 3 4} \, r^2 \, dr \\
    & \leq \frac{C}{\sigma(\sphere)} 
    \int_{\rd 3} V(x)^{\gamma + \frac 3 4} \, dx,
  \end{align*}
  since by H\"older's inequality,
  \begin{equation*}
    \tilde V(r)^{\gamma + \frac 3 4} 
    \leq \frac{1}{\sigma(\sphere)}
    \int_{\sphere} V(r\theta)^{\gamma + \frac 3 4} \,
    d\sigma(\theta),
    \quad r > 0. \qedhere
  \end{equation*}
\end{proof}

We will use the following fourth-order Lieb-Thirring inequality in
$\lspace 2 {\rd 3}$, that follows from more general results in
\cite{netrusov-weidl}. The operator $(-\Delta)^2$ is of course defined as
the self-adjoint operator in $\lspace 2 {\rd 3}$ corresponding to the
closure of the quadratic form
\begin{equation*}
  u \mapsto \int_{\rd 3} \left|\Delta u(x)\right|^2 \, dx,
  \quad u \in \czthree.
\end{equation*}
\begin{lem} \label{lem:nohardy}
  For any $\gamma \geq 1/4$, there is a constant $D = D(\gamma) > 0$ such 
  that for any non-negative $V \in \czthree$, 
  the negative spectrum of $(-\Delta)^2 - V$ is discrete and
  \begin{equation*}
    \tr \left((-\Delta)^2 - V\right)_-^{\gamma}
    \leq D \int_{\rd 3} V(x)^{\gamma + \frac 3 4} \, dx.
  \end{equation*}
\end{lem}

\begin{lem} \label{lem:3d-secondpart}
  Let $\gamma \geq 1/4$ and denote by
  $G_0^{(2)}$ the self-adjoint operator in $\hilbert_2$
  corresponding to the closure $g_0^{(2)}$ of the quadratic form
  \begin{equation*}
    \{u_n\}_{n=1}^\infty \mapsto 
    \sum_{n=1}^\infty \int_0^\infty
    \left|-u_n''(r) + n(n+1)\frac{u_n(r)}{r^2} \right|^2 \, dr,
  \end{equation*}
  initially defined on $D \coloneq P_2 \utrans \czthree$. Then there
  is a constant $C^{(2)} > 0$ such given a non-negative $V \in \czthree$ and 
  $V^{(2)} = P_2 \utrans V \utrans^{-1} P_2$, the
  negative spectrum of $G_0^{(2)} - V^{(2)}$ is discrete and 
  \begin{equation*}
    \tr\left( G_0^{(2)} - V^{(2)}\right)_-^{\gamma}
    \leq C^{(2)} \!\! \int_{\rd 3} V(x)^{\gamma + \frac 3 4} \, dx.
  \end{equation*}
\end{lem}
\begin{proof}
  Consider the operator $(-\Delta)^2$ in $\lspace 2 {\rd 3}$. 
  Let the constant $D > 0$ be as in Lemma \ref{lem:nohardy}.
  By \eqref{eq:laplacedecomp}, the operator 
  $\hat G_0^{(2)} \coloneq \utrans (-\Delta)^2 \utrans^{-1}$ 
  corresponds to the closure in $\hilbert$ of the quadratic form
  \begin{equation*}
    \{u_n\}_{n=0}^\infty \mapsto 
    \sum_{n=0}^\infty \int_0^\infty
    \left|-u_n''(r) + n(n+1)\frac{u_n(r)}{r^2} \right|^2 \, dr,    
  \end{equation*}
  initially defined on $\hat D \coloneq \utrans \czthree$. 

  Choose $V \in \czthree$ and let $V^{(2)} = P_2 \utrans V \utrans^{-1}
  P_2$ and $\hat V^{(2)} = \utrans V
  \utrans^{-1}$. Denote by $E$ and $\hat E$ the spectral measures 
  corresponding to the operators $G_0^{(2)} - V^{(2)}$ and 
  $\hat G_0^{(2)} - \hat V^{(2)}$, respectively. By Glazman's lemma,
  \begin{equation} \label{eq:spectralrankleq}
    \rank E(-\infty,-\lambda) \leq \rank \hat E(-\infty, -\lambda),
  \end{equation}
  for any $\lambda > 0$. Lemma \ref{lem:nohardy} shows that the
  negative spectrum of $\hat G_0^{(2)} - \hat V^{(2)}$ is discrete and
  that
  \begin{equation*}
    \tr\left( \hat G_0^{(2)} - \hat V^{(2)}\right)_-^\gamma
    \leq D \int_{\rd 3} V(x)^{\gamma + \frac 3 4} \, dx.
  \end{equation*}
  The result follows from this and \eqref{eq:spectralrankleq}.
\end{proof}

\begin{prop} \label{prop:niceV-3d}
  Define the quadratic form $h_0$ as the closure of
  \begin{equation*}
    u \mapsto
    \int_{\rd 3}
    \left(
      \left|\Delta u(x)\right|^2
      - \hardyconst{3}{2} \frac{|u(x)|^2}{|x|^4}
    \right) \, dx
  \end{equation*}
  on $\czthree$.
  Let $H_0$ be the self-adjoint operator in $\lspace 2 {\rd 3}$
  corresponding to $h_0$.
  Then there is a constant $C > 0$  such that 
  for any non-negative $V \in \czthree$, the negative spectrum of 
  $H_0-V$ is discrete and
  \begin{equation*}
    \tr (H_0 - V)_-^{1/4} \leq C \int_{\rd 3} V(x) \, dx.
  \end{equation*}
\end{prop}
\begin{proof}
  Let $G_0^{(1)}$, $G_0^{(2)}$, $g_0^{(1)}$, $g_0^{(2)}$, $C^{(1)}$
  and $C^{(2)}$ be as in Lemmas \ref{lem:3d-firstpart} and
  \ref{lem:3d-secondpart}.
  Choose any $u \in \czthree$ and $\epsilon$ with $0 < \epsilon <
  1$. Write $\{u_n\}_{n=0}^\infty = \utrans u$ and note that each
  $u_n \in \cz$. 
  Using \eqref{eq:laplacedecomp} and $\hardyconst{3}{2} = \hardyconst{1}{2}$,
  and finally Lemma \ref{lem:orthhardy}
  it follows that
  \begin{align*}
    h_0[u]
    = & \sum_{n=0}^\infty \int_0^\infty \left(
      \left|-u_n''(r) + n(n+1)\frac{u_n(r)}{r^2} \right|^2
      - \hardyconst{1}{2} \frac{|u_n(r)|^2}{r^4} 
    \right) \, dr \\
    \geq & \int_0^\infty \left(
      \left|u_0''(r)\right|^2
      - \hardyconst{1}{2} \frac{|u_0(r)|^2}{r^4} 
    \right) \, dr \\
    & +  \epsilon \sum_{n=1}^\infty \int_0^\infty
      \left|-u_n''(r) + n(n+1)\frac{u_n(r)}{r^2} \right|^2 \, dr \\
    & +  
    \sum_{n=1}^\infty \int_0^\infty \left(
      (1-\epsilon)
      \left(1 + \hardyconst{1}{2}\right)
      - \hardyconst{1}{2}
    \right) \frac{|u_n(r)|^2}{r^4} \, dr.
  \end{align*}
  Fixing $\epsilon = 1/(1+\hardyconst{1}{2})$ we get that
  \begin{equation*}
    (1-\epsilon)(1+\hardyconst{1}{2}) - \hardyconst{1}{2} = 0
  \end{equation*}
  and thus
  \begin{equation*}
    h_0[u] \geq g_0^{(1)}\left[P_1 \utrans u\right] + 
    \epsilon g_0^{(2)}\left[P_2 \utrans u\right], 
  \end{equation*}
  for any $u \in \czthree$.
  Since $\czthree$ is initial domain of $h_0$ and 
  since the initial domains of $g_0^{(1)}$ and $g_0^{(2)}$ are
  $P_1 \utrans \czthree$ and $P_2 \utrans \czthree$, respectively, it
  follows that
  \begin{equation*}
    \utrans H_0 \utrans^{-1} \geq G_0^{(1)} \oplus \epsilon G_0^{(2)}.
  \end{equation*}

  Choose $V \in \czthree$ and let $W = \utrans V \utrans^{-1}$,
  $V^{(1)} = P_1WP_1$ and $V^{(2)} = P_2WP_2$.
  Since $W$ is bounded and non-negative, it follows for
  any $f \in \hilbert$ that
  \begin{equation*}
    2\realpart (P_1WP_2f,f) 
    \leq 
    2\left\|W^{1/2}P_2f\right\| \left\|W^{1/2}P_1f\right\|
    \leq \left(P_1WP_1f,f\right) 
    + \left(P_2WP_2f,f\right).
  \end{equation*}
  Hence
  \begin{equation*}
    P_1WP_2 + P_2WP_1 \leq V^{(1)} + V^{(2)},
  \end{equation*}
  and therefore
  \begin{align*}
    \utrans H_0 \utrans^{-1} - W 
    & \geq \left(G_0^{(1)} \oplus \epsilon \, G_0^{(2)}\right) - W \\
    & \geq \left(G_0^{(1)}-2 V^{(1)}\right)
    \oplus \epsilon \left(G_0^{(2)} - \frac 2 \epsilon V^{(2)}\right).
  \end{align*}
  The result now follows from Lemmas \ref{lem:3d-firstpart} and
  \ref{lem:3d-secondpart}.
\end{proof}

\section{Proof of the Main Results} \label{sec:mainproof}

We are now in a position to prove Theorems 
\ref{thm:main1d}, \ref{thm:main3d} and \ref{thm:maingeneral}. 
This will be accomplished by approximating general potentials with
smooth, compactly supported ones, as in the abstract lemma below, 
and then combine this 
with Corollary \ref{cor:niceV-general} 
and Propositions \ref{prop:niceV-1d} and \ref{prop:niceV-3d}.

\begin{lem}
  Let $\Omega$ be an open subset of $\rd d$. Suppose that $H_0$ is a
  non-negative symmetric operator in $\lspace 2 \Omega$, 
  defined on $\czeroinf \Omega$. Let
  $\hat{H}_0$ be the Friedrich extension of $H_0$, and choose
  $\gamma_1 \geq 0$, $\gamma_2 \geq 1$ and $\xp \geq 0$. Suppose that
  there is a constant $C > 0$, such that for any non-negative
  $V \in \czeroinf{\Omega}$, the negative spectrum of $\hat H_0 - V$ is
  discrete, and
  \begin{equation*}
    \tr(\hat H_0 - V)_-^{\gamma_1} \leq C \int_\Omega V(x)^{\gamma_2}
    |x|^{\xp} \, dx.
  \end{equation*}
  Then, for any non-negative potential $\hat V$ such that
  $\hat V(x)^{\gamma_2}x^{\xp}$ is integrable on $\Omega$, the
  quadratic form
  \begin{equation*}
    h_{\hat V}[u] = (H_0 u, u)
    - \int_\Omega \hat V(x) |u(x)|^2 \, dx,
    \quad u \in \czeroinf \Omega
  \end{equation*}
  is semi-bounded. Furthermore, if we let $\hat H_0 - \hat V$ be the
  self-adjoint operator associated with the closure of the above form,
  then the negative spectrum of $\hat H_0 - \hat V$ is discrete and
  \begin{equation*}
    \tr \left( \hat H_0 - \hat V \right)_-^{\gamma_1}
    \leq \int_\Omega \hat V(x)^{\gamma_2} |x|^{\xp} \, dx.
  \end{equation*}
\end{lem}
\begin{proof}
  Choose a non-negative $\hat V$ such that 
  \begin{equation*}
    D \coloneq \int_\Omega \hat V(x)^{\gamma_2} |x|^{\xp} \, dx < \infty.
  \end{equation*}
  Also choose a sequence $0 \leq V_1 \leq V_2 \leq \cdots
  \leq \hat V$ such that each $V_n$ belongs to $\czeroinf \Omega$
  and such that for almost any $x \in \Omega$, 
  \begin{equation*}
    V_n(x) \to \hat V(x), \quad \text{ as } n \to \infty.
  \end{equation*}
  
  By assumption, we know that the negative spectrum of $\hat H_0 - V_n$
  is discrete and
  \begin{equation} \label{eq:n-trace}
    \tr (\hat H_0 - V_n)_-^{\gamma_1}
    \leq C \int_\Omega V_n(x)^{\gamma_2} |x|^{\xp} \, dx
    \leq CD.
  \end{equation}
  Hence in particular, 
  \begin{equation*}
    \inf \spec{\hat H_0-V_n} \geq -(CD)^{1/\gamma_1}.
  \end{equation*}
  Note that for given $u \in \czeroinf \Omega$,
  monotone convergence shows that
  \begin{equation*}
    h_{\hat V}[u] = \lim_{n\to\infty} \left( (H_0u,u)
      - \int_\Omega \!\! V_n |u|^2\, dx \right)
  \end{equation*}
  and therefore $h_{\hat V}$ is lower semi-bounded by
  $-(CD)^{1/\gamma_1}$ on $\czeroinf \Omega$. 
  Recall that the operator $\hat H_0 - \hat V$ is
  defined as the self-adjoint operator associated with the closure of 
  $h_{\hat V}$. 

  Let $E$ and $E_n$ be the spectral measures corresponding to 
  $\hat H_0-\hat V$ and $\hat H_0 - V_n$, respectively. For $\lambda > 0$, let
  \begin{equation*}
    N(\lambda) = \rank E(-\infty, -\lambda)
    \quad \text{ and } \quad
    N(\lambda,n) = \rank E_n(-\infty, -\lambda).
  \end{equation*}
  For fixed $\lambda > 0$, Glazman's lemma shows that
  \begin{equation*}
    N(\lambda,1) \leq N(\lambda,2) \leq \cdots \leq N(\lambda),
  \end{equation*}
  and since
  \begin{equation*}
    N(\lambda,n) \cdot \lambda^{\gamma_1} 
    \leq \tr (\hat H_0-V_n)_-^{\gamma_1} \leq CD,
  \end{equation*}
  it must be that 
  \begin{equation*}
    N(\lambda,\infty) \coloneq \lim_{n\to\infty} N(\lambda,n) < \infty.
  \end{equation*}
  Since $N(\lambda,n)$ only assumes integer values, it follows that there is
  an integer $m = m(\lambda) \geq 1$ such that
  \begin{equation*}
    N(\lambda,\infty) = N(\lambda,n),
    \quad n \geq m.
  \end{equation*}
  Let us prove that for any $\lambda > 0$,
  \begin{equation} \label{eq:Neq}
    N(\lambda) = N(\lambda,\infty).
  \end{equation}

  Fix $\lambda > 0$ and let $N = N(\lambda,\infty)$. 
  Assume that $N(\lambda) > N$. Then
  in fact it is possible to find a $\delta > 0$ such that 
  $N(\lambda + \delta) > N$.
  Again using Glazman's
  lemma we find a linear set $F \subset \czeroinf{\Omega}$ with 
  $\dim F = N + 1$ such that
  \begin{equation} \label{eq:hV-ineq}
    h_{\hat V}[f] < -(\lambda+\delta)\|f\|^2
  \end{equation}
  for any $f \in F$ with $f \neq 0$. Let $\{f_1,f_2,\ldots,f_{N+1}\}$
  be an orthonormal basis in $F$. 
  Use monotone convergence to fix $n \geq 1$ such that
  \begin{equation} \label{eq:Vn-approx}
    (N+1) \int_\Omega (\hat V-V_n) |f_k|^2 \, dx < \delta,
    \quad k = 1,2,\ldots,N+1.
  \end{equation}
  Now, since $\dim F > N(\lambda,n)$ there must be scalars
  $c_1,c_2,\ldots,c_{N+1}$, not all zero, such that 
  $g \coloneq c_1f_1 + c_2f_2 + \cdots + c_{N+1}f_{N+1}$ satisfies
  \begin{equation*}
    (H_0 g,g) - \int_\Omega \!\! V_n|g|^2 \, dx \geq -\lambda\|g\|^2.
  \end{equation*}
  Note that $|c_1|^2 + |c_2|^2 + \cdots +
  |c_{N+1}|^2 = \|g\|^2$.
  Now, by \eqref{eq:Vn-approx}
  \begin{align*}
    h_{\hat V}[g] & = (H_0 g, g) - \int_\Omega \!\! V_n|g|^2 \, dx 
    - \int_\Omega (\hat V - V_n)|g|^2 \\
    & \geq -\lambda \|g\|^2 - (N+1)\sum_{k=1}^{N+1} |c_k|^2 
    \int_\Omega (\hat V - V_n) |f_k|^2 \, dx \\
    & \geq - (\lambda  + \delta) \|g\|^2.
  \end{align*}
  This contradicts \eqref{eq:hV-ineq}, and therefore \eqref{eq:Neq}
  holds.  

  By \eqref{eq:Neq}, the negative spectrum of $\hat H - \hat V$ is
  discrete. Denote by $\lambda_j$ and $\lambda_{j,n}$, 
  where $j = 1,2,3,\ldots$,
  the negative eigenvalues of $\hat H_0 - \hat V$ and $\hat H_0 -
  V_n$, respectively, ordered such that
  \begin{equation*}
    \lambda_1 \leq \lambda_2 \leq \cdots
    \quad \text{ and } \quad
    \lambda_{1,n} \leq \lambda_{2,n} \leq \cdots.
  \end{equation*}
  For convenience, we always consider infinite sequences $\lambda_j$
  and $\lambda_{j,n}$. If there should only be finitely many, say $k$,
  negative eigenvalues for the operator $\hat H_0 - \hat V$, or $\hat H_0 -
  V_n$, we let $\lambda_j = 0$, or $\lambda_{j,n} = 0$, for $j > k$.
  Note that 
  \begin{equation*}
    \lambda_{j,1} \geq \lambda_{j,2} \geq \cdots \geq \lambda_j.
  \end{equation*}
  Using \eqref{eq:Neq} again, we see that $\lambda_{j,n} \to \lambda_j$ as 
  $n \to \infty$, and the result now follows from \eqref{eq:n-trace}.
\end{proof}

\noindent \textbf{Acknowledgments}. The authors would like to 
thank Ari Laptev and Imperial College in London for their 
hospitality during the final stages of
writing this paper. Both authors were
partially supported by the ESF program SPECT. Andreas Enblom was also 
supported by grant KAW 2005.0098 from the Knut and Alice Wallenberg
Foundation.


\bibliography{fourth}
\bibliographystyle{alpha}

\end{document}